\documentclass[a4paper]{amsart}
\usepackage{amscd,amsmath,amsthm,amssymb}
\usepackage[all]{xy}

\theoremstyle{plain}
\newtheorem{theo}{Theorem}[section]
\newtheorem{lemm}[theo]{Lemma}
\newtheorem{prop}[theo]{Proposition}
\newtheorem{coro}[theo]{Corollary}

\theoremstyle{definition}
\newtheorem{defi}[theo]{Definition}
\theoremstyle{remark}
\newtheorem{rema}[theo]{Remark}

\numberwithin{equation}{section}

\DeclareMathOperator{\spec}{Spec} 
\newcommand{\iso}{\cong}

\newcommand{\id}{{\rm id}}
\newcommand{\st}{\,|\,} 
\newcommand{\dual}{^{\vee}} 

\newcommand{\mono}{\hookrightarrow} 
\newcommand{\mor}[1]{\xrightarrow{#1}} 
\newcommand{\isomor}{\mor{\sim}} 
\newcommand{\comp}{\circ} 
\newcommand{\sw}[1]{\lvert#1\rvert} 
\newcommand{\rest}[1]{|_{#1}} 
\newcommand{\pr}{\pi} 
\newcommand{\diag}{\delta} 
\newcommand{\ring}[1]{\mathbb{#1}}
\newcommand{\K}{\Bbbk} 
\newcommand{\C}{\ring{C}}

\newcommand{\Z}{\ring{Z}}
\newcommand{\cat}[1]{{\mathfrak{#1}}} 
\newcommand{\opp}{^{\circ}} 
\newcommand{\mrs}{^{\#}} 
\newcommand{\s}[1]{\mathcal{#1}} 
\newcommand{\Hom}{{\rm Hom}} 
\newcommand{\sHom}{\s{H}om} 
\newcommand{\Aut}{{\rm Aut}} 
\newcommand{\Pic}{{\rm Pic}} 
\newcommand{\F}{\mathbb{F}} 
\newcommand{\Ps}{\mathbb{P}} 
\newcommand{\w}{{\rm w}} 
\newcommand{\so}{\s{O}} 
\newcommand{\ds}{\omega} 
\newcommand{\dss}{\ds'} 
\newcommand{\sod}[1][]{\s{O}_{\Delta_{#1}}} 
\newcommand{\cone}[1]{\mathsf{C}(#1)} 
\newcommand{\FM}[1]{\Phi_{#1}} 
\newcommand{\FMa}[1]{\Phi'_{#1}} 
\newcommand{\FMcomp}{\star} 
\newcommand{\farg}{-} 
\newcommand{\trspan}[1]{\langle#1\rangle} 
\newcommand{\tc}[1]{\s{C}_{#1}} 
\newcommand{\D}[1][]{\mathrm{D}^{#1}} 
\newcommand{\siso}[1]{\iota_{#1}} 
\newcommand{\mult}[1]{\mu_{#1}} 
\newcommand{\comult}[1]{\nu_{#1}} 
\newcommand{\multind}[2]{\bar{\mu}_{#2,#1}} 
\newcommand{\ST}[1]{\mathsf{T}_{#1}} 
\newcommand{\M}[1]{\mathsf{M}_{\Z_{#1}}} 
\newcommand{\Ma}[1]{\tilde{\mathsf{M}}_{\Z_{#1}}} 
\newcommand{\MM}[1]{\mathsf{N}_{\Z_{#1}}} 
\newcommand{\T}[1]{\mathsf{L}_{#1}} 
\newcommand{\fun}{\mathsf{F}} 
\newcommand{\funn}{\mathsf{G}} 
\newcommand{\funp}{\tilde{\fun}} 
\newcommand{\fundiag}[1][\Ker]{\mult{#1}} 
\newcommand{\tdual}{^!} 
\newcommand{\Ker}{\s{K}} 
\newcommand{\Kerr}{\s{L}} 
\newcommand{\Kerrr}{\s{M}} 
\newcommand{\Kerd}{\Ker\tdual} 
\newcommand{\Kerrd}{\Kerr\tdual} 
\newcommand{\exc}[1]{\s{E}_{#1}} 
\newcommand{\ldual}[2]{L^{(#1)}#2} 
\newcommand{\excd}[1]{\s{E}'_{#1}} 
\newcommand{\excp}[1]{\tilde{\s{E}}_{#1}} 
\newcommand{\clr}{\cat{R}} 
\newcommand{\crr}{\cat{S}} 
\newcommand{\lr}[1]{\s{R}_{#1}} 
\newcommand{\rr}[1]{\s{S}_{#1}} 
\newcommand{\lrd}[1]{\rho_{#1}} 
\newcommand{\drr}[1]{\sigma_{#1}} 
\newcommand{\rrlr}[1]{\tau_{#1}} 
\newcommand{\lrlr}[1]{\alpha_{#1}} 
\newcommand{\lre}[1]{\beta_{#1}} 
\newcommand{\elr}[1]{\gamma_{#1}} 
\newcommand{\err}[1]{\gamma'_{#1}} 
\newcommand{\flrd}[1]{\tilde{\rho}_{#1}} 
\newcommand{\dc}[1]{\tilde{\sigma}_{#1}} 
\newcommand{\cflr}[1]{\tilde{\tau}_{#1}} 
\newcommand{\frrc}[1]{\zeta_{#1}} 
\newcommand{\fec}[1]{\xi_{#1}} 
\newcommand{\comultind}[2]{\bar{\nu}_{#1,#2}} 
\newcommand{\pt}{\{\cdot\}} 
\newcommand{\radj}{^!} 
\newcommand{\hcomp}{\cdot} 
\newcommand{\ob}{{\rm Ob}} 
\newcommand{\ort}{^{\perp}} 

\begin{document}

\title[Exceptional sequences and derived autoequivalences]{
Exceptional sequences and derived autoequivalences}

\author[A. Canonaco]{Alberto Canonaco}
\address{Dipartimento di Matematica ``F. Casorati'', Universit{\`a}
di Pavia, Via Ferrata 1, 27100 Pavia, Italy}
\email{alberto.canonaco@unipv.it}



\begin{abstract}
We prove a general theorem that gives a non trivial relation in the
group of derived autoequivalences of a variety (or stack) $X$, under
the assumption that there exists a suitable functor from the derived
category of another variety $Y$ admitting a full exceptional
sequence. Applications include the case in which $X$ is Calabi-Yau and
either $X$ is a hypersurface in $Y$ (this extends a previous result by
the author and R.~L. Karp, where $Y$ was a weighted projective space)
or $Y$ is a hypersurface in $X$.  The proof uses a resolution of the
diagonal of $Y$ constructed from the exceptional sequence.
\end{abstract}

\maketitle

\section{Introduction}\label{intro}

If $X$ is a smooth proper variety or stack, the group $\Aut(\D[b](X))$
of (isomorphism classes of) exact autoequivalences of $\D[b](X)$ (the
bounded derived category of coherent sheaves on $X$) is an interesting
object of study, in particular when $X$ is Calabi-Yau, i.e. when
$\ds_X\iso\so_X$ and $H^i(X,\so_X)=0$ for $0<i<\dim(X)$.  Obviously in
any case the following elements are in $\Aut(\D[b](X))$: shift
functors $(\farg)[n]$ for every integer $n$; pull-back functors $f^*$
for $f$ an automorphism of $X$; functors $\T{\s{F}}$, defined as
$\s{F}\otimes\farg$, for $\s{F}$ a line bundle on $X$. In fact, if
$\ds_X$ or $\ds_X^{-1}$ is ample, then these elements generate
$\Aut(\D[b](X))$, which is isomorphic to
$\Z\times(\Pic(X)\rtimes\Aut(X))$ (see \cite{BO}). On the other hand,
in general $\Aut(\D[b](X))$ is much bigger and its structure is rather
mysterious. However, it is known that, at least if $X$ is a smooth
projective variety (see \cite{O}) or more generally the smooth stack
associated to a normal projective variety with only quotient
singularities (see \cite{Ka}), then every autoequivalence of
$\D[b](X)$ is a {\em Fourier-Mukai functor}. An exact functor
$F\colon\D[b](Y)\to\D[b](X)$ is a Fourier-Mukai functor if there
exists an object $\Ker\in\D[b](X\times Y)$ (called {\em kernel} of
$F$) such that $F\iso\FM{\Ker}$, where $\FM{\Ker}$ denotes the functor
${\pr_1}_*(\Ker\otimes\pr_2^*(\farg))$.

Interesting examples of non trivial elements in $\Aut(\D[b](X))$ were
introduced in \cite{ST}, where it was proved that the Fourier-Mukai
functor $\ST{\s{F}}:=\FM{\cone{\s{F}\boxtimes\s{F}\dual\to\sod}}$
(where $\sod$ is the structure sheaf of the diagonal in $X\times X$,
the morphism is the natural one and $\cone{\farg}$ denotes the cone of
a morphism) is an autoequivalence of $\D[b](X)$ when
$\s{F}\in\D[b](X)$ is a {\em spherical object}, i.e. when
$\s{F}\otimes\ds_X\iso\s{F}$ and
\[\Hom_{\D[b](X)}(\s{F},\s{F}[k])\iso\begin{cases}
\K & \text{if $k=0,\dim(X)$} \\
0 & \text{otherwise}
\end{cases}\]
($\ST{\s{F}}$ is then called the {\em spherical twist} associated to
$\s{F}$). Clearly every line bundle on $X$ is a spherical object if
$X$ is Calabi-Yau. Spherical objects can also be constructed in some
common geometric settings starting with another smooth proper variety
(or stack) $Y$ and an {\em exceptional object} $\exc{}\in\D[b](Y)$,
i.e. such that
\[\Hom_{\D[b](Y)}(\exc{},\exc{}[k])\iso\begin{cases}
\K & \text{if $k=0$} \\
0 & \text{otherwise}
\end{cases}\]
(see \cite[Chapter 8]{H} for an account on spherical and exceptional
objects and relations between them).  For instance, if $f\colon X\mono
Y$ is the inclusion of a hypersurface such that $\ds_Y\iso\so_Y(-X)$,
then $f^*\exc{}\in\D[b](X)$ is spherical; also, if $g\colon Y\mono X$
is the inclusion of a hypersurface such that $g^*\ds_X\iso\so_Y$, then
$g_*\exc{}\in\D[b](X)$ is spherical. The aim of this paper is to find,
in similar situations, relations in $\Aut(\D[b](X))$ between the
spherical twists associated to the images of a {\em full exceptional
sequence} $(\exc{0},\dots,\exc{m})$ in $\D[b](Y)$ (this means that
each $\exc{i}$ is exceptional, that
$\Hom_{\D[b](Y)}(\exc{i},\exc{j}[k])=0$ for $i>j$ and for every
$k\in\Z$, and that $\D[b](Y)$ is generated by
$\{\exc{0},\dots,\exc{m}\}$ as a triangulated category). We were
motivated by the search for a general framework allowing to formulate
and prove relations in $\Aut(\D[b](X))$, including as particular cases
those proved in \cite{CKa} and \cite{K1} and those conjectured in
\cite{K1} and \cite{K2}, which we now recall briefly.

In \cite{CKa} it was proved that if $X\subset\Ps=\Ps(\w_0,\dots,\w_n)$
is a hypersurface of degree $\sw{\w}:=\w_0+\cdots+\w_n$ (hence $X$ is
Calabi-Yau), then the autoequivalence
$\funn:=\T{\so_X(1)}\comp\ST{\so_X}$ of $\D[b](X)$ satisfies
$\funn^{\comp\sw{\w}}\iso(\farg)[2]$. This is easily seen to be
equivalent to $\ST{\so_X(1)}\comp\cdots\comp\ST{\so_X(\sw{\w})}\iso
\T{\so_X(-\sw{\w})[2]}$, and, since
$(\so_{\Ps}(1),\dots,\so_{\Ps}(\sw{\w}))$ is a full exceptional
sequence in $\D[b](\Ps)$, it is natural to conjecture that more
generally
\begin{equation}\label{epb}
\ST{f^*\exc{0}}\comp\cdots\comp\ST{f^*\exc{m}}\iso\T{\so_X(-X)[2]}
\end{equation}
if, as above, $f\colon X\mono Y$ is the inclusion of a hypersurface
such that $\ds_Y\iso\so_Y(-X)$.

On the other hand, some of the relations in \cite{K1} and \cite{K2}
(proved when $Y$ is $\Ps^1$ and conjectured when $Y$ is $\Ps^2$ or the
Hirzebruch surface $\F_3$) are equivalent to particular cases of the
following statement: if $g\colon Y\mono X$ is the inclusion of a
hypersurface such that $g^*\ds_X\iso\so_Y$, then
\begin{equation}\label{epo}
\ST{g_*\exc{0}}\comp\cdots\comp\ST{g_*\exc{m}}\iso\T{\so_X(Y)}.
\end{equation}

In this paper we prove both \eqref{epb} and \eqref{epo} (cor.
\ref{pullback} and \ref{pushout}) as consequences of the following
more general result (theorem \ref{mthm}): given a Fourier-Mukai
functor $\fun\iso\FM{\Ker}\colon\D[b](Y)\to\D[b](X)$ satisfying
condition \eqref{maincond} (which essentially says that $F$ acts fully
faithfully on the non trivial part of the morphisms between the terms
of the exceptional sequence), we have
\begin{equation}\label{emt}
\ST{\fun(\exc{0})}\comp\cdots\comp\ST{\fun(\exc{m})}\iso
\FM{\cone{\fundiag}},
\end{equation}
where $\fundiag$ is a morphism in $\D[b](X\times X)$ naturally induced
by $\Ker$ (see section \ref{comp} for the precise definition). Here it
is enough to say that one can construct a functor
$\funp\colon\D[b](Y\times Y)\to\D[b](X\times X)$ and that $\fundiag$
can be identified with a natural morphism $\funp(\sod[Y])\to\sod[X]$.

Our strategy of proof is similar to that of \cite[Theorem 1.1]{CKa}
and is based on the use of a suitable resolution of the diagonal for
$Y$. Namely, denoting by $(\excd{m},\dots,\excd{0})$ the dual
exceptional sequence of $(\exc{0},\dots\exc{m})$ (characterized
by $\Hom_{\D[b](Y)}(\exc{i},\excd{j}[k])\iso
\K^{\delta_{i,j}\delta_{k,0}}$ for $0\le i,j\le m$ and for every
$k\in\Z$), objects $\lr{k}\in\D[b](Y\times Y)$ for $0\le k\le m$ are
defined with the following properties:
\begin{enumerate}
\item $\lr{0}\iso\excd{0}\boxtimes\exc{0}\dual$;
\item there is a distinguished triangle
$\lr{k-1}\mor{\lrlr{k}}\lr{k}\mor{\lre{k}}
\excd{k}\boxtimes\exc{k}\dual\mor{\elr{k}}\lr{k-1}[1]$ for $0<k\le m$;
\item $\lr{m}\iso\sod$.
\end{enumerate}
In \cite{CKa}, where $Y=\Ps$ and $\exc{i}=\so_{\Ps}(i)$, this is
achieved by induction on $k$, defining explicitly the morphism
$\elr{k}$ (the notation of \cite{CKa} is actually different), and then
the difficult part is to show that (3) holds. This approach seems hard
to follow in the general case. Instead, using suitable semiorthogonal
decompositions of $\D[b](Y\times Y)$, in section \ref{resol} we can
define directly objects $\lr{k}$ such that (3) is automatically
satisfied, and then prove that (1) and (2) hold using the fact that
there is a natural way to define $\lrlr{k}$. Then the idea for the
proof of \eqref{emt} is simply to show by induction on $k$ that
$\ST{\fun(\exc{0})}\comp\cdots\comp\ST{\fun(\exc{k})}\iso
\FM{\cone{\funp(\lr{k})\to\sod[X]}}$. Here some of the difficulties
come from the fact that at several points one has to check that the
morphisms involved are the ``right'' ones. In most cases this is just
a matter of checking that some natural diagrams commute: to this
purpose in section \ref{comp} we prove some general properties about
compositions of kernels of Fourier-Mukai functors, which allow to get
in an efficient way the needed commutative diagrams. A more delicate
problem is that, due to the non functoriality of the cone, not all the
morphisms are natural: to deal with this question we use a property
(proved in \cite{CKu} and recalled in the appendix) of triangulated
categories satisfying a suitable finiteness condition.

\subsection*{Acknowledgements}
It is a pleasure to thank Robert L. Karp for useful conversations and
for drawing my attention to his conjectures.

\subsection*{Notation}
We work over a fixed base field $\K$.  As in \cite{CKa} for simplicity
we will call {\em stack} a connected Deligne-Mumford stack which is
smooth and proper over $\K$, and such that every coherent sheaf is a
quotient of a locally free sheaf of finite rank. 

We will write $\pt$ for the stack $\spec\K$.  If $X$ is a stack,
$\D[b](X)$ will denote the bounded derived category of coherent
sheaves on $X$. Our definition of stack implies that tensor product
admits a left derived functor
$\farg\otimes\farg\colon\D[b](X)\times\D[b](X)\to\D[b](X)$ and that
for every morphism of stacks $f\colon X\to Y$ there are (left and
right) derived functors $f^*\colon\D[b](Y)\to\D[b](X)$ and
$f_*\colon\D[b](X)\to\D[b](Y)$ (we use this notation since we never
need to distinguish between these functors and the corresponding
underived versions). Also the functors $\Hom_X$ and $\sHom_X$ admit
derived functors
$\Hom_X(\farg,\farg)\colon\D[b](X)\opp\times\D[b](X)\to\D[b](\pt)$
(hence with this convention
\[\Hom_X(\s{F},\s{F}')\iso
\bigoplus_{k\in\Z}\Hom_{\D[b](X)}(\s{F},\s{F}'[k])[-k]\]
for $\s{F},\s{F}'\in\D[b](X)$) and
$\sHom_X(\farg,\farg)\colon\D[b](X)\opp\times\D[b](X)\to\D[b](X)$; for
$\s{F}\in\D[b](X)$ we set $\s{F}\dual:=\sHom_X(\s{F},\so_X)$.

For a stack $X$, $\diag\colon X\to X\times X$ will be the diagonal
morphism; we will often write $\sod[X]$ (or simply $\sod$) instead of
$\diag_*\so_X$.  Denoting as usual by $\ds_X$ the dualizing sheaf on
$X$, we define $\dss_X:=\ds_X[\dim(X)]$.  For a morphism of stacks
$f\colon X\to Y$, $f^!\colon\D[b](Y)\to\D[b](X)$ denotes the functor
$f^*(\farg)\otimes\dss_X\otimes f^*{\dss_Y}\dual$.

If $X$ and $Y$ are two stacks, we will denote by $\pr_1\colon X\times
Y\to X$ and $\pr_2\colon X\times Y\to Y$ the projections and by
$\farg\boxtimes\farg$ the exterior (derived) tensor
product\footnote{In \cite{CKa} the opposite convention
$\farg\boxtimes\farg:=\pr_2^*(\farg)\otimes\pr_1^*(\farg)$ is used.}
\[\pr_1^*(\farg)\otimes\pr_2^*(\farg)
\colon\D[b](X)\times\D[b](Y)\to\D[b](X\times Y).\]

We refer to the appendix for notations and conventions about
triangulated categories and exact functors between them.

\section{Composition of kernels}\label{comp}

$X$, $Y$, $Z$ and $W$ will be stacks.
For $\Ker\in\D[b](X\times Y)$ and $\Kerr\in\D[b](Y\times Z)$, we set
\[\Ker\FMcomp\Kerr:=
{\pr_{1,3}}_*(\pr_{1,2}^*\Ker\otimes\pr_{2,3}^*\Kerr)\in
\D[b](X\times Z)\] 
(where $\pi_{i,j}$ denotes the projection from $X\times Y\times Z$
onto the $i^{\rm th}\times j^{\rm th}$ factor). Notice that, given
$\s{F}\in\D[b](X)$ and $\s{G}\in\D[b](Y)$, identifying $X$ with
$X\times\pt$ and $Y$ with $\pt\times Y$, $\s{F}\FMcomp\s{G}$
coincides with $\s{F}\boxtimes\s{G}$. 

Clearly for every $\Ker\in\D[b](X\times Y)$ and for every $Z$ there
are exact functors
\[\Ker\FMcomp\farg\colon\D[b](Y\times Z)\to\D[b](X\times Z)
\text{\qquad and\qquad}
\farg\FMcomp\Ker\colon\D[b](Z\times X)\to\D[b](Z\times Y).\]
In particular, when $Z=\pt$, $\Ker\FMcomp\farg$, respectively
$\farg\FMcomp\Ker$ can be identified with the Fourier-Mukai functor
with kernel $\Ker$ 
\[{\pr_1}_*(\Ker\otimes\pr_2^*(\farg))\colon\D[b](Y)\to\D[b](X),
\text{ respectively }
{\pr_2}_*(\Ker\otimes\pr_1^*(\farg))\colon\D[b](X)\to\D[b](Y),\]
which will be denoted by $\FM{\Ker}$, respectively $\FMa{\Ker}$.

The following result shows that, up to isomorphism, the operation
$\FMcomp$ is associative and the objects $\sod$ act as identities:
since we do not know a reference for this fact, we include a proof
here.

\begin{lemm}\label{assoc}
For every $\Ker\in\D[b](X\times Y)$ there are natural isomorphisms
$\sod[X]\FMcomp\Ker\iso\Ker\iso\Ker\FMcomp\sod[Y]$. Moreover, given
also $\Kerr\in\D[b](Y\times Z)$ and
$\Kerrr\in\D[b](Z\times W)$, there is a natural isomorphism
$(\Ker\FMcomp\Kerr)\FMcomp\Kerrr\iso\Ker\FMcomp(\Kerr\FMcomp\Kerrr)$
in $\D[b](X\times W)$.
\end{lemm}

\begin{proof}
Applying flat base change to the cartesian square
\[\xymatrix{X\times Y \ar[rr]^-{\tilde{\diag}}
\ar[d]_{\pr_1} & & X\times X\times Y \ar[d]^{\pr_{1,2}} \\
X \ar[rr]_-{\diag} & & X\times X}\]
and using projection formula for $\tilde{\diag}$, we obtain
\[\sod[X]\FMcomp\Ker=
{\pr_{1,3}}_*(\pr_{1,2}^*\diag_*\so_X\otimes\pr_{2,3}^*\Ker)\iso
{\pr_{1,3}}_*(\tilde{\diag}_*\so_{X\times Y}\otimes\pr_{2,3}^*\Ker)
\iso{\pr_{1,3}}_*\tilde{\diag}_*\tilde{\diag}^*\pr_{2,3}^*\Ker,\]
and the last term is isomorphic to $\Ker$ because
$\pr_{1,3}\comp\tilde{\diag}=\pr_{2,3}\comp\tilde{\diag}=
\id_{X\times Y}$. In a completely similar way one proves that also
$\Ker\FMcomp\sod[Y]\iso\Ker$.

In order to prove the second statement, we will denote by $\pr_{i,j}$
and $\pr_{i,j,k}$ the obvious projections from $X\times Y\times
Z\times W$ and, for $V$ one of $X$, $Y$, $Z$ and $W$, by $\pr^V_{i,j}$
the obvious projections from the product of all the four terms except
$V$. Applying flat base change to the cartesian square
\[\xymatrix{X\times Y\times Z\times W \ar[rr]^-{\pr_{1,3,4}}
\ar[d]_{\pr_{1,2,3}} & & X\times Z\times W \ar[d]^{\pr^Y_{1,2}} \\
X\times Y\times Z \ar[rr]_-{\pr^W_{1,3}} & & X\times Z,}\]
using projection formula for $\pr_{1,3,4}$ and taking into account
that $\pr^Y_{1,3}\comp\pr_{1,3,4}=\pr_{1,4}$,
$\pr^W_{1,2}\comp\pr_{1,2,3}=\pr_{1,2}$,
$\pr^W_{2,3}\comp\pr_{1,2,3}=\pr_{2,3}$ and
$\pr^Y_{2,3}\comp\pr_{1,3,4}=\pr_{3,4}$, we obtain
\begin{multline*}
(\Ker\FMcomp\Kerr)\FMcomp\Kerrr=
{\pr^Y_{1,3}}_*({\pr^Y_{1,2}}^*
{\pr^W_{1,3}}_*({\pr^W_{1,2}}^*\Ker\otimes{\pr^W_{2,3}}^*\Kerr)
\otimes{\pr^Y_{2,3}}^*\Kerrr) \\
\iso{\pr^Y_{1,3}}_*({\pr_{1,3,4}}_*
{\pr_{1,2,3}}^*({\pr^W_{1,2}}^*\Ker\otimes{\pr^W_{2,3}}^*\Kerr)
\otimes{\pr^Y_{2,3}}^*\Kerrr) \\
\iso{\pr^Y_{1,3}}_*{\pr_{1,3,4}}_*
({\pr_{1,2,3}}^*({\pr^W_{1,2}}^*\Ker\otimes{\pr^W_{2,3}}^*\Kerr)
\otimes{\pr_{1,3,4}}^*{\pr^Y_{2,3}}^*\Kerrr) \\
\iso{\pr_{1,4}}_*(\pr_{1,2}^*\Ker\otimes\pr_{2,3}^*\Kerr\otimes
\pr_{3,4}^*\Kerrr).
\end{multline*}
In a completely similar way one proves that also
\[\Ker\FMcomp(\Kerr\FMcomp\Kerrr)\iso{\pr_{1,4}}_*
(\pr_{1,2}^*\Ker\otimes\pr_{2,3}^*\Kerr\otimes\pr_{3,4}^*\Kerrr).\]
\end{proof}

\begin{rema}\label{FMcomp}
When $W=\pt$ the above result yields the well known fact (see e.g.
\cite[Prop. 5.10]{H}) that 
$\FM{\Ker}\comp\FM{\Kerr}\iso\FM{\Ker\FMcomp\Kerr}\colon
\D[b](Z)\to\D[b](X)$.
\end{rema}

From \ref{assoc} it is easy to deduce various results (like
\cite[Lemma 2.2]{CKa}) about Fourier-Mukai functors with kernels of the
form $\s{F}\boxtimes\s{G}$. In particular we will need the following fact.

\begin{coro}\label{splitker}
Given $\s{F},\s{F}'\in\D[b](X)$ and $\s{G},\s{G}'\in\D[b](Y)$, there are
natural isomorphisms $\FM{\s{F}\boxtimes\s{G}}(\s{G}')\iso
\s{F}\otimes_{\K}\Hom_Y(\s{G}\dual,\s{G}')$ and
$\FMa{\s{F}\boxtimes\s{G}}(\s{F}')\iso
\s{G}\otimes_{\K}\Hom_X(\s{F}\dual,\s{F}')$.
\end{coro}

\begin{proof}
We have $\FM{\s{F}\boxtimes\s{G}}(\s{G}')\iso
(\s{F}\FMcomp\s{G})\FMcomp\s{G}'\iso\s{F}\FMcomp(\s{G}\FMcomp\s{G}')$,
where we consider $\s{G}\in\D[b](\pt\times X)$ and
$\s{G}'\in\D[b](X\times\pt)$; it is then enough to note that in this
case
$\s{G}\FMcomp\s{G}'\iso\Hom_Y(\s{G}\dual,\s{G}')\in\D[b](\pt)$. The
other proof is similar.
\end{proof}

From now on by abuse of notation we will often treat as equalities the
natural isomorphisms given by \ref{assoc}.

For $\Ker\in\D[b](X\times Y)$ we set
$\Kerd:=\Ker\dual\otimes\pr_2^*\dss_Y$, but considering it as an
object of $\D[b](Y\times X)$. Notice that, given $\s{F}\in\D[b](X)$,
identifying $X$ with $X\times\pt$, $\s{F}\tdual$ coincides with
$\s{F}\dual$.

\begin{lemm}\label{adjoint}
For every $\Ker\in\D[b](X\times Y)$ and for every other stack $Z$ the
functor
\[\Ker\FMcomp\farg\colon\D[b](Y\times Z)\to\D[b](X\times Z),
\text{ respectively }
\farg\FMcomp\Ker\colon\D[b](Z\times X)\to\D[b](Z\times Y)\]
is left, respectively right adjoint of the functor
\[\Kerd\FMcomp\farg\colon\D[b](X\times Z)\to\D[b](Y\times Z),
\text{ respectively }
\farg\FMcomp\Kerd\colon\D[b](Z\times Y)\to\D[b](Z\times X).\]
\end{lemm}

\begin{proof}
By definition $\Ker\FMcomp\farg$ is the composition of the functors
$\pr_{2,3}^*$, $\pr_{1,2}^*\Ker\otimes\farg$ and
${\pr_{1,3}}_*$. Remembering that, if $f$ is a morphism of stacks, the
right adjoints of $f^*$ and $f_*$ are respectively $f_*$ and $f^!$,
and that the right adjoint of $\s{F}\otimes\farg$ is
$\s{F}\dual\otimes\farg$, we conclude that the right adjoint of
$\Ker\FMcomp\farg$ is
\begin{multline*}
{\pr_{2,3}}_*((\pr_{1,2}^*\Ker)\dual\otimes\pr_{1,3}^!(\farg))\iso
{\pr_{2,3}}_*(\pr_{1,2}^*\Ker\dual\otimes\dss_{X\times Y\times Z}
\otimes(\pr_{1,3}^*\dss_{X\times Z})\dual\otimes\pr_{1,3}^*(\farg))\\
\iso{\pr_{2,3}}_*(\pr_{1,2}^*\Kerd\otimes\pr_{1,3}^*(\farg))\iso
\Kerd\FMcomp\farg.
\end{multline*}
In a completely similar way one proves that the left ajoint of
$\farg\FMcomp\Ker$ is $\farg\FMcomp\Kerd$.
\end{proof}

\begin{rema}\label{FMadj}
When $Z=\pt$ the above result says that
$\FM{\Ker}\colon\D[b](Y)\to\D[b](X)$, respectively
$\FMa{\Ker}\colon\D[b](X)\to\D[b](Y)$ is left, respectively right
adjoint of $\FM{\Kerd}\colon\D[b](X)\to\D[b](Y)$, respectively
$\FMa{\Kerd}\colon\D[b](Y)\to\D[b](X)$ (see e.g. \cite[Prop. 5.9]{H}).
\end{rema}

For every $\Ker\in\D[b](X\times Y)$ by \ref{adjoint} there are natural
isomorphisms
\[\Hom_{\D[b](X\times X)}(\Ker\FMcomp\Kerd,\sod[X])\iso
\Hom_{\D[b](X\times Y)}(\Ker,\Ker)\iso
\Hom_{\D[b](Y\times Y)}(\sod[Y],\Kerd\FMcomp\Ker),\]
and we denote by $\mult{\Ker}\colon\Ker\FMcomp\Kerd\to\sod[X]$ and 
$\comult{\Ker}\colon\sod[Y]\to\Kerd\FMcomp\Ker$ the morphisms
corresponding to $\id_{\Ker}$. In particular, given
$\s{F}\in\D[b](X)$, identifying as usual $X$ with $X\times\pt$ and
$\s{F}\boxtimes\s{F}\dual$ with $\s{F}\FMcomp\s{F}\tdual$,
$\mult{\s{F}}$ coincides with the natural morphism
$\s{F}\boxtimes\s{F}\dual\to\sod[X]$. Given also
$\s{A}\in\D[b](X\times X)$, we define
\begin{equation}\label{multind}
\multind{\s{A}}{\s{F}}\colon\FM{\s{A}}(\s{F})\boxtimes\s{F}\dual\iso
\s{A}\FMcomp(\s{F}\boxtimes\s{F}\dual)
\mor{\id_{\s{A}}\FMcomp\mult{\s{F}}}\s{A}
\end{equation}
Clearly the morphism $\multind{\s{A}}{\s{F}}$ is functorial in
$\s{A}$, meaning that for every morphism $\alpha\colon\s{A}\to\s{A}'$
in $\D[b](X\times X)$ there is a commutative diagram
\begin{equation}\label{natmultind}
\xymatrix{\FM{\s{A}}(\s{F})\boxtimes\s{F}\dual
\ar[rr]^-{\multind{\s{A}}{\s{F}}}
\ar[d]_{\FM{\alpha}(\s{F})\boxtimes\id} & &
\s{A} \ar[d]^{\alpha} \\
\FM{\s{A}'}(\s{F})\boxtimes\s{F}\dual
\ar[rr]_-{\multind{\s{A}'}{\s{F}}} & & \s{A}'.}
\end{equation}

\begin{prop}\label{compat}
Given $\Ker\in\D[b](X\times Y)$ and $\Kerr\in\D[b](Y\times Z)$, we
have:
\begin{enumerate}
\item There is a natural isomorphism $(\Ker\FMcomp\Kerr)\tdual\iso
\Kerrd\FMcomp\Kerd$ in $\D[b](Z\times X)$.
\item The diagram
\[\xymatrix{
\Ker\FMcomp\Kerr\FMcomp\Kerrd\FMcomp\Kerd \ar[rr]^-{\sim}
\ar[d]_{\id_{\Ker}\FMcomp\mult{\Kerr}\FMcomp\id_{\Kerd}} & &
(\Ker\FMcomp\Kerr)\FMcomp(\Ker\FMcomp\Kerr)\tdual
\ar[d]^{\mult{\Ker\FMcomp\Kerr}} \\
\Ker\FMcomp\Kerd \ar[rr]_-{\mult{\Ker}} & & \sod[X]
}\]
(where the top horizontal isomorphism is given by (1)) commutes.
\item The compositions
\begin{gather*}
\Ker\mor{\id_{\Ker}\FMcomp\comult{\Ker}}
\Ker\FMcomp\Kerd\FMcomp\Ker\mor{\mult{\Ker}\FMcomp\id_{\Ker}}\Ker, \\
\Kerd\mor{\comult{\Ker}\FMcomp\id_{\Kerd}}
\Kerd\FMcomp\Ker\FMcomp\Kerd\mor{\id_{\Kerd}\FMcomp\mult{\Ker}}\Kerd
\end{gather*}
coincide with $\id_{\Ker}$ and $\id_{\Kerd}$.
\item The diagram
\[\xymatrix{
\Kerr\FMcomp\Kerrd\FMcomp\Kerd \ar[rr]^-{\sim}
\ar[d]_{\mult{\Kerr}\FMcomp\id_{\Kerd}} & &
\Kerr\FMcomp(\Ker\FMcomp\Kerr)\tdual
\ar[d]^{\comult{\Ker}\FMcomp\id_{\Kerr\FMcomp(\Ker\FMcomp\Kerr)\tdual}} \\
\Kerd & & \Kerd\FMcomp\Ker\FMcomp\Kerr\FMcomp(\Ker\FMcomp\Kerr)\tdual
\ar[ll]^-{\id_{\Kerd}\FMcomp\mult{\Ker\FMcomp\Kerr}}
}\]
(where the top horizontal isomorphism is given by (1)) commutes.
\end{enumerate}
\end{prop}

\begin{proof}
Denoting by $\fun\radj$ the right adjoint of a functor $\fun$, it is clear
that if $\fun$ and $\funn$ are two composable functors admitting a right
adjoint, then $(\fun\comp\funn)\radj\iso\funn\radj\comp\fun\radj$. Evaluating
this isomorphism at $\sod[X]$ when
\begin{gather*}
\fun=\Ker\FMcomp\farg\colon\D[b](Y\times X)\to\D[b](X\times X), \\
\funn=\Kerr\FMcomp\farg\colon\D[b](Z\times X)\to\D[b](Y\times X)
\end{gather*}
and using \ref{adjoint} we obtain (1). Similarly, denoting by
$\mult{\fun}\colon\fun\comp\fun\radj\to\id$ and
$\comult{\fun}\colon\id\to\fun\radj\comp\fun$ the adjunction morphisms,
(2) follows from the fact that the diagram
\[\xymatrix{
\fun\comp\funn\comp\funn\radj\comp\fun\radj \ar[rr]^-{\sim}
\ar[d]_{\fun\hcomp\mult{\funn}\hcomp\fun\radj} & &
(\fun\comp\funn)\comp(\fun\comp\funn)\radj
\ar[d]^{\mult{\fun\comp\funn}} \\
\fun\comp\fun\radj \ar[rr]_-{\mult{\fun}} & & \id
}\]
commutes (see \cite[Theorem 1, p. 101]{M}), whereas (3) follows from
the fact that the compositions $\fun\mor{\id\hcomp\comult{\fun}}
\fun\comp\fun\radj\comp\fun\mor{\mult{\fun}\hcomp\id}\fun$ and
$\fun\radj\mor{\comult{\fun}\hcomp\id}\fun\radj\comp\fun\comp\fun\radj
\mor{\id\hcomp\mult{\fun}}\fun\radj$ coincide with $\id_{\fun}$ and
$\id_{\fun\radj}$ (see \cite[Theorem 1, p. 80]{M}). As for (4), in the
diagram
\[\xymatrix{
\Kerr\FMcomp\Kerrd\FMcomp\Kerd \ar[rrrr]^-{\sim}
\ar[rd]^{\comult{\Ker}\FMcomp\id_{\Kerr\FMcomp\Kerrd\FMcomp\Kerd}}
\ar[dd]^{\mult{\Kerr}\FMcomp\id_{\Kerd}} & & & &
\Kerr\FMcomp(\Ker\FMcomp\Kerr)\tdual
\ar[dd]_{\comult{\Ker}\FMcomp\id_{\Kerr\FMcomp(\Ker\FMcomp\Kerr)\tdual}} \\
 & \Kerd\FMcomp\Ker\FMcomp\Kerr\FMcomp\Kerrd\FMcomp\Kerd
\ar[drrr]^{\sim}
\ar[d]_{\id_{\Kerd\FMcomp\Ker}\FMcomp\mult{\Kerr}\FMcomp\id_{\Kerd}} \\
\Kerd \ar[r]_-{\comult{\Ker}\FMcomp\id_{\Kerd}} &
\Kerd\FMcomp\Ker\FMcomp\Kerd \ar[r]_-{\id_{\Kerd}\FMcomp\mult{\Ker}}
& \Kerd & & \Kerd\FMcomp\Ker\FMcomp\Kerr\FMcomp(\Ker\FMcomp\Kerr)\tdual
\ar[ll]^-{\id_{\Kerd}\FMcomp\mult{\Ker\FMcomp\Kerr}}
}\]
the triangle commutes by (2). Since the other two inner quadrangles
clearly commute as well, we conclude that the outer square also
commutes, which yields the result, taking into account (3).
\end{proof}

For the rest of this section we fix $\Ker\in\D[b](X\times Y)$ and we
set
\begin{gather}
\fun:=\FM{\Ker}\colon\D[b](Y)\to\D[b](X),\nonumber \\
\funp:=\Ker\FMcomp\farg\FMcomp\Kerd\colon\D[b](Y\times Y)\to
\D[b](X\times X).\label{funp}
\end{gather}
Note that $\funp(\sod[Y])$ is naturally isomorphic to
$\Ker\FMcomp\Kerd$, and we will freely regard $\mult{\Ker}$ as a
morphism $\fundiag\colon\funp(\sod[Y])\to\sod[X]$.

\begin{coro}\label{Fprod}
For every $\s{G},\s{G}'\in\D[b](Y)$ there is a natural isomorphism
\[\funp(\s{G}'\boxtimes\s{G}\dual)\iso
\fun(\s{G}')\boxtimes\fun(\s{G})\dual.\]
\end{coro}

\begin{proof}
By part (1) of \ref{compat} we have
\[\funp(\s{G}'\boxtimes\s{G}\dual)\iso
(\Ker\FMcomp\s{G}')\FMcomp(\s{G}\tdual\FMcomp\Kerd)\iso
(\Ker\FMcomp\s{G}')\FMcomp(\Ker\FMcomp\s{G})\tdual\iso
\fun(\s{G}')\boxtimes\fun(\s{G})\dual.\]
\end{proof}

\begin{coro}\label{natmult}
For every $\s{G}\in\D[b](Y)$ the diagram
\[\xymatrix{
\funp(\s{G}\boxtimes\s{G}\dual) \ar[rr]^-{\sim}
\ar[d]_{\funp(\mult{\s{G}})} & &
\fun(\s{G})\boxtimes\fun(\s{G})\dual
\ar[d]^{\mult{\fun(\s{G})}} \\
\funp(\sod[Y]) \ar[rr]_-{\fundiag} & & \sod[X]
}\]
commutes, where the top horizontal map is the natural isomorphism
given by \ref{Fprod}.
\end{coro}

\begin{proof}
It follows immediately from part (2) of \ref{compat}.
\end{proof}

Given $\s{A}\in\D[b](Y\times Y)$ and $\s{G}\in\D[b](Y)$, we define the
natural morphism
\begin{equation}\label{comultind}
\comultind{\s{A}}{\s{G}}\colon\fun(\FM{\s{A}}(\s{G}))\iso
\Ker\FMcomp\s{A}\FMcomp\s{G}
\mor{\id_{\Ker\FMcomp\s{A}}\FMcomp\comult{\Ker}\FMcomp\id_{\s{G}}}
\Ker\FMcomp\s{A}\FMcomp\Kerd\FMcomp\Ker\FMcomp\s{G}
\iso\FM{\funp(\s{A})}(\fun(\s{G})),
\end{equation}
which is functorial in $\s{A}$, meaning that for every morphism
$\alpha\colon\s{A}\to\s{A}'$ in $\D[b](Y\times Y)$ there is a
commutative diagram in $\D[b](X)$
\begin{equation}\label{natcomultind}
\xymatrix{\fun(\FM{\s{A}}(\s{G})) \ar[rr]^-{\comultind{\s{A}}{\s{G}}}
\ar[d]_{\fun(\FM{\alpha}(\s{G}))} & & \FM{\funp(\s{A})}(\fun(\s{G}))
\ar[d]^{\FM{\funp(\alpha)}(\fun(\s{G}))} \\
\fun(\FM{\s{A}'}(\s{G})) \ar[rr]_-{\comultind{\s{A}'}{\s{G}}} & &
\FM{\funp(\s{A}')}(\fun(\s{G})).}
\end{equation}

\begin{coro}\label{split}
For every $\s{G}\in\D[b](Y)$ the composition
\[\fun(\s{G})\iso\fun(\FM{\sod[Y]}(\s{G}))\mor{\comultind{\sod[Y]}{\s{G}}}
\FM{\funp(\sod[Y])}(\fun(\s{G}))\mor{\FM{\fundiag}(\fun(\s{G}))}
\FM{\sod[X]}(\fun(\s{G}))\iso\fun(\s{G})\]
coincides with $\id_{\fun(\s{G})}$.
\end{coro}

\begin{proof}
It is enough to observe that the above sequence can be identified with
the image through the functor $\farg\FMcomp\s{G}$ of the first
sequence in part (3) of \ref{compat}.
\end{proof}

\begin{coro}\label{multcomult}
For every $\s{A}\in\D[b](Y\times Y)$ and for every $\s{G}\in\D[b](Y)$
the diagram
\[\xymatrix{\funp(\FM{\s{A}}(\s{G})\boxtimes\s{G}\dual)
\ar[rr]^-{\sim} \ar[d]_{\funp(\multind{\s{A}}{\s{G}})} & & 
\fun(\FM{\s{A}}(\s{G}))\boxtimes\fun(\s{G})\dual 
\ar[d]^{\comultind{\s{A}}{\s{G}}\boxtimes\id} \\
\funp(\s{A}) & & 
\FM{\funp(\s{A})}(\fun(\s{G}))\boxtimes\fun(\s{G})\dual
\ar[ll]^-{\multind{\funp(\s{A})}{\fun(\s{G})}}}\] 
commutes, where the top horizontal map is the natural isomorphism
given by \ref{Fprod}.
\end{coro}

\begin{proof}
It is immediate to see that the above diagram can be identified with
the image through the functor $(\Ker\FMcomp\s{A})\FMcomp\farg$ of the
commutative diagram in part (4) of \ref{compat}, with $\s{G}$ in place
of $\Kerr$.
\end{proof}

\section{Resolution of the diagonal via a full exceptional
sequence}\label{resol}

Let $Y$ be a stack and assume that $(\exc{0},\dots,\exc{m})$ is a full
exceptional sequence in $\D[b](Y)$; we will denote by
$(\excd{m},\dots,\excd{0})$ its dual (full exceptional) sequence (see
the appendix for its definition).

For $0\le k\le m$ we define $\clr_k:=
\trspan{\exc{0}\dual,\dots,\exc{k}\dual}$ and $\crr_k:=
\trspan{\exc{k+1}\dual,\dots,\exc{m}\dual}$ (they are admissible
subcategories of $\D[b](Y)$, as clearly
$(\exc{m}\dual,\dots,\exc{0}\dual)$ is a full exceptional sequence in
$\D[b](Y)$). If $\cat{C}$ is a subcategory of $\D[b](Y)$, we set
\[\D[b](Y)\boxtimes\cat{C}:=
\trspan{\{\s{F}\boxtimes\s{G}\st\s{F}\in\D[b](Y),\s{G}\in\cat{C}\}}.\]
It follows from \cite[prop. 2.1.18]{B} that
$(\D[b](Y)\boxtimes\crr_k,\D[b](Y)\boxtimes\clr_k)$ is
a semiorthogonal decomposition of $\D[b](Y\times Y)$. Hence by
\ref{semiort} there exists (unique up to isomorphism) a distinguished
triangle in $\D[b](Y\times Y)$
\[\lr{k}\mor{\lrd{k}}\sod\mor{\drr{k}}\rr{k}\mor{\rrlr{k}}\lr{k}[1]\]
with $\lr{k}\in\D[b](Y)\boxtimes\clr_k$ and $\rr{k}\in
\D[b](Y)\boxtimes\crr_k$ (in particular, $\rr{m}=0$ and
$\lrd{m}$ is an isomorphism). For $0<k\le m$, since
$\Hom_{Y\times Y}(\lr{k-1},\rr{k})=0$, there exists a unique morphism
$\lrlr{k}\colon\lr{k-1}\to\lr{k}$ such that $\lrd{k-1}=\lrd{k}\comp\lrlr{k}$.

\begin{prop}\label{diagres}
$\lr{0}\iso\excd{0}\boxtimes\exc{0}\dual=\exc{0}\boxtimes\exc{0}\dual$
and $\cone{\lrlr{k}}\iso \excd{k}\boxtimes\exc{k}\dual$ for $0<k\le m$.
\end{prop}

\begin{proof}
Setting $\excp{0}:=\lr{0}$ and $\excp{k}:=\cone{\lrlr{k}}$ for $0<k\le
m$, first we claim that $\excp{k}\iso\s{F}_k\boxtimes\exc{k}\dual$ for
some $\s{F}_k\in\D[b](Y)$. Indeed, for $k=0$ this is clear, and for
$0<k\le m$ it follows from the fact that $\excp{k}$ belongs both to
$\D[b](Y)\boxtimes\clr_k$ (because $\lrlr{k}$ is a morphism of
$\D[b](Y)\boxtimes\clr_k$) and to $\D[b](Y)\boxtimes\crr_{k-1}$
(because by (TR4) there is a distinguished triangle
\[\cone{\lrlr{k}}=\excp{k}\to\cone{\lrd{k-1}}\iso\rr{k-1}\to
\cone{\lrd{k}}\iso\rr{k}\to\cone{\lrlr{k}}[1]\]
and $\rr{k-1},\rr{k}\in\D[b](Y)\boxtimes\crr_{k-1}$).
Therefore by \ref{splitker} we have
\begin{equation}\label{FME}
\FMa{\excp{k}}(\exc{i}\dual)\iso
\exc{k}\dual\otimes_{\K}\Hom_Y(\exc{i},\s{F}_k)
\end{equation}
for $0\le i,k\le m$. On the other hand, from the distinguished
triangle
\[\FMa{\lr{k}}(\exc{i}\dual)\mor{\FMa{\lrd{k}}(\exc{i}\dual)}
\FMa{\sod}(\exc{i}\dual)\iso\exc{i}\dual\to
\FMa{\rr{k}}(\exc{i}\dual)\to\FMa{\lr{k}}(\exc{i}\dual)[1],\]
since clearly $\FMa{\lr{k}}(\exc{i}\dual)\in\clr_k$ and
$\FMa{\rr{k}}(\exc{i}\dual)\in\crr_k$, we deduce that
$\FMa{\lrd{k}}(\exc{i}\dual)$ is an isomorphism for $i\le k$, while
$\FMa{\lr{k}}(\exc{i}\dual)=0$ for $i>k$. Thus for $0<k\le m$ the
equality $\FMa{\lrd{k-1}}(\exc{i}\dual)=\FMa{\lrd{k}}(\exc{i}\dual)
\comp\FMa{\lrlr{k}}(\exc{i}\dual)$ implies that
$\FMa{\lrlr{k}}(\exc{i}\dual)$ is an isomorphism for $i<k$. Hence from
the distinguished triangle
\[\FMa{\excp{k}}(\exc{i}\dual)[-1]\to\FMa{\lr{k-1}}(\exc{i}\dual)
\mor{\FMa{\lrlr{k}}(\exc{i}\dual)}\FMa{\lr{k}}(\exc{i}\dual)\to
\FMa{\excp{k}}(\exc{i}\dual)\]
we get $\FMa{\excp{k}}(\exc{i}\dual)\iso(\exc{k}\dual)^{\delta_{i,k}}$
(notice that, by what we have already proved, this is true also for
$k=0$), which is equivalent, by \eqref{FME}, to
\begin{equation}\label{Eort}
\Hom_Y(\exc{i},\s{F}_k)\iso\K^{\delta_{i,k}}
\end{equation}
for $0\le i,k\le m$. In order to conclude that
$\s{F}_k\iso\excd{k}$ one can proceed as follows. As
$(\excd{m},\dots,\excd{0})$ is a full exceptional sequence in
$\D[b](Y)$, there exists (unique up to isomorphism) a distinguished
triangle
\[\s{A}_k\to\s{F}_k\to\s{A}'_k\to\s{A}_k[1]\]
with $\s{A}_k\in\trspan{\excd{0},\dots,\excd{k-1}}$ and
$\s{A}'_k\in\trspan{\excd{k},\dots,\excd{m}}$. Now, if $i\ge
k$, $\Hom_Y(\exc{i},\s{A}_k)=0$ by \ref{ort}, whereas if $i<k$,
$\Hom_Y(\exc{i},\s{A}'_k)=0$ always by \ref{ort} and
$\Hom_Y(\exc{i},\s{F}_k)=0$ by \eqref{Eort}, so that
$\Hom_Y(\exc{i},\s{A}_k)=0$ also in this case. Clearly this implies
$\s{A}_k=0$ and $\s{F}_k\iso\s{A}'_k\in
\trspan{\excd{k},\dots,\excd{m}}$. In a similar way, there
exists (unique up to isomorphism) a distinguished triangle
\[\s{B}_k\to\s{F}_k\to\s{B}'_k\to\s{B}_k[1]\]
with $\s{B}_k\in\trspan{\excd{k}}$ and $\s{B}'_k\in
\trspan{\excd{k+1},\dots,\excd{m}}$. This time
$\Hom_Y(\exc{i},\s{B}'_k)=0$ (if $i\le k$ by \ref{ort}, if $i>k$
because $\Hom_Y(\exc{i},\s{B}_k)=0$ by \ref{ort} and
$\Hom_Y(\exc{i},\s{F}_k)=0$ by \eqref{Eort}), whence $\s{B}'_k=0$ and
$\s{F}_k\iso\s{B}_k\in\trspan{\excd{k}}$. Then, since
$\Hom_Y(\exc{k},\s{F}_k)\iso\K\iso\Hom_Y(\exc{k},\excd{k})$, we must
have $\s{F}_k\iso\excd{k}$.
\end{proof}

Hence for $0<k\le m$ there is a distinguished triangle in
$\D[b](Y\times Y)$
\begin{equation}\label{diagtr}
\lr{k-1}\mor{\lrlr{k}}\lr{k}\mor{\lre{k}}\excp{k}\mor{\elr{k}}\lr{k-1}[1]
\end{equation}
(where $\excp{k}:=\excd{k}\boxtimes\exc{k}\dual$). Notice that the
pair $(\lre{k},\elr{k})$ is not uniquely determined, but, as clearly
$\excp{k}$ is an exceptional object, it could only be changed by
$(\lambda\lre{k},\lambda^{-1}\elr{k})$ for some $\lambda\in\K^*$. By
(TR4) (or by (TR3) and \ref{cone!}) there exists (unique because
$\Hom(\lr{k-1}[1],\rr{k-1})=0$) a morphism
$\err{k}\colon\excp{k}\to\rr{k-1}$ such that the diagram
\[\xymatrix{\lr{k-1} \ar[rr]^-{\lrlr{k}}
\ar[d]^{\id} & & \lr{k} \ar[rr]^-{\lre{k}}
\ar[d]^{\lrd{k}} & &
\excp{k} \ar[rr]^-{\elr{k}} \ar[d]^{\err{k}} & &
\lr{k-1}[1] \ar[d]^{\id} \\
\lr{k-1} \ar[rr]^-{\lrd{k-1}} & & \sod \ar[rr]^-{\drr{k-1}} & &
\rr{k-1} \ar[rr]^-{\rrlr{k-1}} & & \lr{k-1}[1]}\]
commutes and with $\cone{\err{k}}\iso\rr{k}$.

\begin{lemm}\label{EL}
For $0<k\le m$ there is an isomorphism
$\excd{k}\iso\FM{\rr{k-1}}(\exc{k})$ with the property that the
induced morphism
\[\excp{k}=\excd{k}\boxtimes\exc{k}\dual\iso
\FM{\rr{k-1}}(\exc{k})\boxtimes\exc{k}\dual
\mor{\multind{\rr{k-1}}{\exc{k}}}\rr{k-1}\]
(see \eqref{multind} for the definition of
$\multind{\rr{k-1}}{\exc{k}}$) coincides with $\err{k}$.
\end{lemm}

\begin{proof}
Since $\FM{\excp{i}}(\exc{k})\iso
\excd{i}\otimes_{\K}\Hom_Y(\exc{i},\exc{k})$ by \ref{splitker}, we
have $\FM{\excp{i}}(\exc{k})=0$ for $i>k$, while there is a natural
isomorphism $\FM{\excp{k}}(\exc{k})\iso\excd{k}$. It follows easily
that $\FM{\rr{k}}(\exc{k})=0$, hence from the distinguished triangle
\[\FM{\excp{k}}(\exc{k})\mor{\FM{\err{k}}(\exc{k})}
\FM{\rr{k-1}}(\exc{k})\to\FM{\rr{k}}(\exc{k})\to\FM{\excp{k}}(\exc{k})[1]\]
we see that $\FM{\err{k}}(\exc{k})$ is an isomorphism. We claim that
the isomorphism we are looking for can be chosen to be
$\excd{k}\iso\FM{\excp{k}}(\exc{k})\mor{\FM{\err{k}}(\exc{k})}
\FM{\rr{k-1}}(\exc{k})$. Indeed, this amounts to say that the diagram
\[\xymatrix{\excp{k}=\excd{k}\boxtimes\exc{k}\dual \ar[rr]^-{\sim} 
\ar[drr]_{\id} & & \FM{\excp{k}}(\exc{k})\boxtimes\exc{k}\dual
\ar[rr]^-{\FM{\err{k}}(\exc{k})\boxtimes\id}
\ar[d]^{\multind{\excp{k}}{\exc{k}}} & &
\FM{\rr{k-1}}(\exc{k})\boxtimes\exc{k}\dual
\ar[d]^{\multind{\rr{k-1}}{\exc{k}}} \\
 & & \excp{k} \ar[rr]_-{\err{k}} & & \rr{k-1}}\]
commutes. Now, it is not difficult to check directly that the triangle
commutes, whereas the square commutes by \eqref{natmultind}.
\end{proof}

\section{Main Theorem}\label{main}

Let $X$ and $Y$ be stacks and $\Ker$ an object of $\D[b](X\times
Y)$. As in section \ref{resol} (whose notation will be used) we assume
that $\D[b](Y)$ admits a full exceptional sequence
$(\exc{0},\dots,\exc{m})$, and moreover that
$\fun:=\FM{\Ker}\colon\D[b](Y)\to\D[b](X)$ induces isomorphisms
\begin{equation}\label{maincond}
\Hom_Y(\exc{i},\exc{j})\isomor\Hom_X(\fun(\exc{i}),\fun(\exc{j}))
\text{ for }0\le i<j\le m.
\end{equation}
Defining $\funp$ as in \eqref{funp}, for $0\le k\le m$ we set
$\flrd{k}:=\fundiag\comp\funp(\lrd{k})\colon\funp(\lr{k})\to\sod[X]$
and extend it to a distinguished triangle
$\funp(\lr{k})\mor{\flrd{k}}\sod[X]
\mor{\dc{k}}\cone{\flrd{k}}\mor{\cflr{k}}\funp(\lr{k})[1]$. Let
moreover $\frrc{k}\colon\funp(\rr{k})\to\cone{\flrd{k}}$ be a morphism
such that the diagram
\[\xymatrix{\funp(\lr{k}) \ar[rr]^-{\funp(\lrd{k})}
\ar[d]^{\id} & & \funp(\sod[Y]) \ar[rr]^-{\funp(\drr{k})}
\ar[d]^{\fundiag} & &
\funp(\rr{k}) \ar[rr]^-{\siso{\funp}(\lr{k})\comp\funp(\rrlr{k})} 
\ar[d]^{\frrc{k}} & & \funp(\lr{k})[1] \ar[d]^{\id} \\
\funp(\lr{k}) \ar[rr]_-{\flrd{k}} & &
\sod[X] \ar[rr]_-{\dc{k}} & &
\cone{\flrd{k}} \ar[rr]_-{\cflr{k}} & &
\funp(\lr{k})[1]}\]
commutes (such a morphism exists by (TR3) or by (TR4)).

\begin{lemm}\label{key}
If $\fun=\FM{\Ker}$ satisfies \eqref{maincond}, then the composition
\[\fun(\FM{\rr{k}}(\exc{k+1}))\mor{\comultind{\rr{k}}{\exc{k+1}}}
\FM{\funp(\rr{k})}(\fun(\exc{k+1}))\mor{\FM{\frrc{k}}(\fun(\exc{k+1}))}
\FM{\cone{\flrd{k}}}(\fun(\exc{k+1}))\]
(see \eqref{comultind} for the definition of
$\comultind{\rr{k}}{\exc{k+1}}$) is an isomorphism in $\D[b](X)$ for
$0\le k<m$.
\end{lemm}

\begin{proof}
Setting for brevity $G(\farg):=\fun(\FM{\farg}(\exc{k+1}))$,
$H(\farg):=\FM{\farg}(\fun(\exc{k+1}))$ and $\tilde{H}:=H\comp\funp$,
there is a commutative diagram (by \eqref{natcomultind})
\[\xymatrix{G(\lr{k}) \ar[r]^-{G(\lrd{k})}
\ar[d]^{\comultind{\lr{k}}{\exc{k+1}}} &
G(\sod[Y]) \ar[r]^-{G(\drr{k})} \ar[d]^{\comultind{\sod[Y]}{\exc{k+1}}} &
G(\rr{k}) \ar[rr]^-{\siso{G}(\lr{k})\comp G(\rrlr{k})} 
\ar[d]^{\comultind{\rr{k}}{\exc{k+1}}} & & G(\lr{k})[1]
\ar[d]^{\comultind{\lr{k}}{\exc{k+1}}[1]} \\
\tilde{H}(\lr{k}) \ar[r]_-{\tilde{H}(\lrd{k})} \ar[d]^{\id} &
\tilde{H}(\sod[Y]) \ar[r]_-{\tilde{H}(\drr{k})} \ar[d]^{H(\fundiag)} &
\tilde{H}(\rr{k}) \ar[d]^{H(\frrc{k})}
\ar[rr]_-{\siso{\tilde{H}}(\lr{k})\comp\tilde{H}(\rrlr{k})} & &
\tilde{H}(\lr{k})[1] \ar[d]^{\id} \\
\tilde{H}(\lr{k}) \ar[r]_-{H(\flrd{k})} &
H(\sod[X]) \ar[r]_-{H(\dc{k})} & H(\cone{\flrd{k}}) 
\ar[rr]_-{\siso{H}(\funp(\lr{k}))\comp H(\cflr{k})} & &
\tilde{H}(\lr{k})[1]}\]
whose rows are distinguished triangles.
As $H(\fundiag)\comp\comultind{\sod[Y]}{\exc{k+1}}=
\FM{\fundiag}(\fun(\exc{k+1}))\comp\comultind{\sod[Y]}{\exc{k+1}}$ is an
isomorphism by \ref{split},
$H(\frrc{k})\comp\comultind{\rr{k}}{\exc{k+1}}=
\FM{\frrc{k}}(\fun(\exc{k+1}))\comp\comultind{\rr{k}}{\exc{k+1}}$ is an
isomorphism if (and only if) $\comultind{\lr{k}}{\exc{k+1}}$ is
an isomorphism. In fact it can be proved that
$\comultind{\lr{i}}{\exc{k+1}}$ is an isomorphism for $0\le i\le k$. To
this purpose it is enough to prove that $\comultind{\excp{i}}{\exc{k+1}}$
is an isomorphism for $0\le i\le k$, because then, remembering that
$\lr{0}\iso\excp{0}$ by \ref{diagres}, one can proceed by induction on
$i$, using the commutative diagrams (whose rows are distinguished
triangles by \eqref{diagtr})
\[\xymatrix{G(\lr{i-1}) \ar[rr]^-{G(\lrlr{i})}
\ar[d]_{\comultind{\lr{i-1}}{\exc{k+1}}} & & G(\lr{i}) 
\ar[rr]^-{G(\lre{i})} \ar[d]_{\comultind{\lr{i}}{\exc{k+1}}} & &
G(\excp{i}) \ar[rr]^-{\siso{G}(\lr{i-1})\comp G(\elr{i})} 
\ar[d]_{\comultind{\excp{i}}{\exc{k+1}}} & &
G(\lr{i-1})[1] \ar[d]_{\comultind{\lr{i-1}}{\exc{k+1}}[1]} \\
\tilde{H}(\lr{i-1}) \ar[rr]_-{\tilde{H}(\lrlr{i})} & & \tilde{H}(\lr{i})
\ar[rr]_-{\tilde{H}(\lre{i})} & & \tilde{H}(\excp{i})
\ar[rr]_-{\siso{\tilde{H}}(\lr{i-1})\comp\tilde{H}(\elr{i})} & &
\tilde{H}(\lr{i-1})[1]}\]
for $0<i\le k$. In order to prove that
\[\comultind{\excp{i}}{\exc{k+1}}\colon\fun(\FM{\excp{i}}(\exc{k+1}))
\to\FM{\funp(\excp{i})}(\fun(\exc{k+1}))\]
is an isomorphism for $0\le i\le k$, notice that by \ref{splitker}
\[\fun(\FM{\excp{i}}(\exc{k+1}))\iso
\fun(\excd{i}\otimes_{\K}\Hom_Y(\exc{i},\exc{k+1}))\iso
\fun(\excd{i})\otimes_{\K}\Hom_Y(\exc{i},\exc{k+1})\]
and (since $\funp(\excp{i})\iso
\fun(\excd{i})\boxtimes\fun(\exc{i})\dual$ by \ref{Fprod})
\[\FM{\funp(\excp{i})}(\fun(\exc{k+1}))\iso
\fun(\excd{i})\otimes_{\K}\Hom_X(\fun(\exc{i}),\fun(\exc{k+1})).\]
It is then easy to see that $\comultind{\excp{i}}{\exc{k+1}}$ can be
identified with the natural map, hence it is an isomorphism by
\eqref{maincond}.
\end{proof}

For $\s{F}\in\D[b](X)$ we set $\ST{\s{F}}:=
\FM{\cone{\mult{\s{F}}\colon\s{F}\boxtimes\s{F}\dual\to\sod}}
\colon\D[b](X)\to\D[b](X)$ (by \cite{ST} $\ST{\s{F}}$ is an
equivalence if $\s{F}$ is a spherical object).

\begin{theo}\label{mthm}
If $\fun=\FM{\Ker}\colon\D[b](Y)\to\D[b](X)$ satisfies
\eqref{maincond}, then
\[\ST{\fun(\exc{0})}\comp\cdots\comp\ST{\fun(\exc{m})}\iso
\FM{\cone{\fundiag\colon\funp(\sod[Y])\to\sod[X]}}.\]
\end{theo}

\begin{proof}
Defining for $0\le k\le m$
\begin{equation*}
\tc{k}:=\cone{\mult{\fun(\exc{0})}}\FMcomp\cdots\FMcomp
\cone{\mult{\fun(\exc{k})}},
\end{equation*}
by \ref{FMcomp} it is enough to prove that
$\tc{m}\iso\cone{\fundiag}$. We will show that in fact
$\tc{k}\iso\cone{\flrd{k}}$ for $0\le k\le m$: the case $k=m$ yields
the thesis because $\flrd{m}=\fundiag\comp\funp(\lrd{m})$ and
$\funp(\lrd{m})$ is an isomorphism, whence
$\cone{\flrd{m}}\iso\cone{\fundiag}$.  We proceed by induction on $k$:
the case $k=0$ follows from \ref{natmult}, since $\mult{\exc{0}}$ can
be identified with $\lrd{0}$. So let $0<k\le m$ and assume that
$\tc{k-1}\iso\cone{\flrd{k-1}}$; then by definition
\begin{multline*}
\tc{k}=\tc{k-1}\FMcomp\cone{\mult{\fun(\exc{k})}\colon 
\fun(\exc{k})\boxtimes\fun(\exc{k})\dual\to\sod[X]} \\
\iso\cone{\id\FMcomp\mult{\fun(\exc{k})}\colon
\cone{\flrd{k-1}}\FMcomp(\fun(\exc{k})\boxtimes\fun(\exc{k})\dual)\to
\cone{\flrd{k-1}}\FMcomp\sod[X]} \\
\iso\cone{\multind{\cone{\flrd{k-1}}}{\fun(\exc{k})}\colon
\FM{\cone{\flrd{k-1}}}(\fun(\exc{k}))\boxtimes\fun(\exc{k})\dual\to
\cone{\flrd{k-1}}}.
\end{multline*}
Denoting by $\fec{k}\colon\funp(\excp{k})\to\cone{\flrd{k-1}}$ the
composition
\begin{multline*}
\funp(\excp{k})\iso
\funp(\FM{\rr{k-1}}(\exc{k})\boxtimes\exc{k}\dual)\iso
\fun(\FM{\rr{k-1}}(\exc{k}))\boxtimes\fun(\exc{k})\dual \\
\mor{\comultind{\rr{k-1}}{\exc{k}}\boxtimes\id}
\FM{\funp(\rr{k-1})}(\fun(\exc{k}))\boxtimes\fun(\exc{k})\dual \\
\mor{\FM{\frrc{k-1}}(\fun(\exc{k}))\boxtimes\id}
\FM{\cone{\flrd{k-1}}}(\fun(\exc{k}))\boxtimes\fun(\exc{k})\dual
\mor{\multind{\cone{\flrd{k-1}}}{\fun(\exc{k})}}\cone{\flrd{k-1}}
\end{multline*}
(the isomorphisms are the natural ones induced by \ref{EL} and
\ref{Fprod}), and remembering that the composition
$\FM{\frrc{k-1}}(\fun(\exc{k}))\comp\comultind{\rr{k-1}}{\exc{k}}$ is an
isomorphism by \ref{key}, we have
$\tc{k}\iso\cone{\multind{\cone{\flrd{k-1}}}{\fun(\exc{k})}}\iso
\cone{\fec{k}}$. Thus the conclusion will follow if we show that the
diagram
\[\xymatrix{\funp(\lr{k-1}) \ar[rr]^-{\funp(\lrlr{k})}
\ar[d]^{\id} & & \funp(\lr{k}) \ar[rr]^-{\funp(\lre{k})}
\ar[d]^{\flrd{k}} & &
\funp(\excp{k}) \ar[rr]^-{\siso{\funp}(\lr{k-1})\comp\funp(\elr{k})} 
\ar[d]^{\fec{k}} & & \funp(\lr{k-1})[1] \ar[d]^{\id} \\
\funp(\lr{k-1}) \ar[rr]_-{\flrd{k-1}} & &
\sod[X] \ar[rr]_-{\dc{k-1}} & &
\cone{\flrd{k-1}} \ar[rr]_-{\cflr{k-1}} & &
\funp(\lr{k-1})[1]}\]
commutes, because then $\cone{\fec{k}}\iso\cone{\flrd{k}}$ by
\ref{cone!}. As the diagram
\[\xymatrix{\funp(\lr{k}) \ar[rrr]^-{\funp(\lre{k})}
\ar[d]^{\funp(\lrd{k})} \ar@(l,l)[dd]_{\flrd{k}} & & &
\funp(\excp{k}) \ar[rrr]^-{\siso{\funp}(\lr{k-1})\comp\funp(\elr{k})} 
\ar[d]^{\funp(\err{k})} & & & \funp(\lr{k-1})[1] \ar[d]^{\id} \\
\funp(\sod[Y]) \ar[rrr]^-{\funp(\drr{k-1})}
\ar[d]^{\fundiag} & & &
\funp(\rr{k-1}) \ar[rrr]^-{\siso{\funp}(\lr{k-1})\comp\funp(\rrlr{k-1})} 
\ar[d]^{\frrc{k-1}} & & & \funp(\lr{k-1})[1] \ar[d]^{\id} \\
\sod[X] \ar[rrr]^-{\dc{k-1}} & & &
\cone{\flrd{k-1}} \ar[rrr]^-{\cflr{k-1}} & & &
\funp(\lr{k-1})[1]}\]
commutes by definition of $\err{k}$ and $\frrc{k-1}$, it is clearly
enough to prove that $\fec{k}=\frrc{k-1}\comp\funp(\err{k})$. This
amounts to say that in the diagram
\[\xymatrix{\funp(\excp{k}) \ar[r]^-{\sim} \ar[d]_{\funp(\err{k})} 
\ar@{}[dr]|(.3){(1)} &
\funp(\FM{\rr{k-1}}(\exc{k})\boxtimes\exc{k}\dual) \ar[r]^-{\sim}
\ar[dl]^{\ \ \funp(\multind{\rr{k-1}}{\exc{k}})} &
\fun(\FM{\rr{k-1}}(\exc{k}))\boxtimes\fun(\exc{k})\dual
\ar[d]^{\comultind{\rr{k-1}}{\exc{k}}\boxtimes\id} \ar@{}[dll]|{(2)} \\
\funp(\rr{k-1}) \ar[d]_{\frrc{k-1}} & & 
\FM{\funp(\rr{k-1})}(\fun(\exc{k}))\boxtimes\fun(\exc{k})\dual
\ar[ll]^-{\multind{\funp(\rr{k-1})}{\fun(\exc{k})}} 
\ar[d]^{\FM{\frrc{k-1}}(\fun(\exc{k}))\boxtimes\id} \ar@{}[dll]|{(3)} \\
\cone{\flrd{k-1}} & &
\FM{\cone{\flrd{k-1}}}(\fun(\exc{k}))\boxtimes\fun(\exc{k})\dual
\ar[ll]^-{\multind{\cone{\flrd{k-1}}}{\fun(\exc{k})}}}\]
the outer square commutes, which is true because (1), (2) and (3)
commute, respectively, by \ref{EL}, \ref{multcomult} and
\eqref{natmultind}.
\end{proof}

\begin{rema}
Theorem \ref{mthm} can be generalized as follows (see also
\cite[Remark 4.7]{CKa}). Given $0\le n\le m$, we set
$\cat{T}:=\trspan{F(\exc{n+1}),\dots,F(\exc{m})}\ort$. Then, assuming
that \eqref{maincond} holds for $0\le i<j\le n$, one can prove that
\begin{equation*}
(\ST{\fun(\exc{0})}\comp\cdots\comp\ST{\fun(\exc{n})})\rest{\cat{T}}
\iso\FM{\cone{\fundiag}}\rest{\cat{T}}.
\end{equation*}
(clearly \ref{mthm} is the case $n=m$). Indeed, with the same proof
one can show $\tc{k}\iso\cone{\flrd{k}}$ for $0\le k\le n$, and then
it is enough to show that $\FM{\cone{\flrd{n}}}\rest{\cat{T}}\iso
\FM{\cone{\fundiag}}\rest{\cat{T}}$. As
$\flrd{n}=\fundiag\comp\funp(\lrd{n})$, this follows from the fact
that $\FM{\cone{\funp(\lrd{n})}}\rest{\cat{T}}\iso0$, which is true
because $\cone{\funp(\lrd{n})}\iso\funp(\cone{\lrd{n}})\iso
\funp(\rr{n})$ and $\FM{\funp(\rr{n})}(\s{F})=0$ for
$\s{F}\in\cat{T}$. 
\end{rema}

\section{Applications}\label{appl}

Condition \eqref{maincond} is satisfied for instance if $\fun$ is fully
faithful, but in that case \ref{mthm} is not so useful. Indeed,
$(\fun(\exc{0}),\dots,\fun(\exc{m}))$ is then an exceptional sequence
in $\D[b](X)$. Taking into account that, if $\exc{}$ is an exceptional
object, $\ST{\exc{}}$ is just projection onto $\trspan{\exc{}}\ort$, it
is clear that $\ST{\fun(\exc{0})}\comp\cdots\comp\ST{\fun(\exc{m})}$
is projection onto $\trspan{\fun(\exc{0}),\dots,\fun(\exc{m})}\ort$,
and it is easy to see directly that the same is true for
$\FM{\cone{\mult{\Ker}}}$.

We are more interested in the following examples.
\begin{enumerate}
\item $\fun=f^*$, where $f\colon X\to Y$ is a morphism such that
$\cone{f\mrs\colon\so_Y\to f_*\so_X}\iso\ds_Y[c]$ where
$c=\dim(Y)-\dim(X)$; in particular, $f$ can be the inclusion of a
hypersurface such that $\ds_Y\iso\so_Y(-X)$ (hence $\ds_X\iso\so_X$).
\item $\fun=g_*$, where $g\colon Y\mono X$ is the inclusion of a
hypersurface such that $g^*\ds_X\iso\so_Y$ (hence
$\ds_Y\iso\so_Y(Y)$);
\end{enumerate}
The fact that \eqref{maincond} holds can be checked as follows. In
case (1), since
\[\Hom_X(f^*\exc{i},f^*\exc{j})\iso
\Hom_Y(\exc{i},f_*f^*\exc{j})\iso
\Hom_Y(\exc{i},\exc{j}\otimes f_*\so_X),\]
and taking into account that $\exc{j}\otimes f_*\so_X\iso
\cone{\exc{j}\otimes\ds_Y[c-1]\to\exc{j}}$ by hypothesis, it is enough
to note that for $i<j$ by Serre duality
\[\Hom_Y(\exc{i},\exc{j}\otimes\dss_Y)\iso
\Hom_Y(\exc{j},\exc{i})\dual=0.\]
Similarly, in case (2) we have
\[\Hom_X(g_*\exc{i},g_*\exc{j})\iso\Hom_Y(g^*g_*\exc{i},\exc{j})\]
so that, since $\cone{g^*g_*\exc{i}\to\exc{i}}\iso
\exc{i}\otimes\so_Y(-Y)[2]$ (by \cite[Cor. 11.4]{H}), again we
conclude from the fact that for $i<j$
\[\Hom_Y(\exc{i}\otimes\so_Y(-Y),\exc{j})\iso
\Hom_Y(\exc{i},\exc{j}\otimes\ds_Y)\iso
\Hom_Y(\exc{j},\exc{i}[\dim(Y)])\dual=0.\]

\begin{rema}
In case (1), if moreover $\ds_X\iso\so_X$, the $f^*\exc{i}$ are
spherical objects by \cite{ST} or \cite[Prop. 8.39]{H}. Similarly, in
case (2) the $g_*\exc{i}$ are spherical objects thanks to
\cite[Prop. 3.15]{ST}.
\end{rema}

In the following, given $\s{F}\in\D[b](X)$, we will denote by
$\T{\s{F}}$ the exact functor
$\FM{\diag_*\s{F}}\iso\s{F}\otimes\farg\colon\D[b](X)\to\D[b](X)$
(which is an equivalence when $\s{F}$ is a line bundle).

\begin{coro}\label{pullback}
If $f\colon X\mono Y$ is the inclusion of a hypersurface such that
$\ds_Y\iso\so_Y(-X)$, then
$\ST{f^*\exc{0}}\comp\cdots\comp\ST{f^*\exc{m}}\iso\T{\so_X(-X)[2]}$.
\end{coro}

\begin{proof}
$f^*\iso\FM{\Ker}$, where $\Ker:=(\id_X,f)_*\so_X\in\D[b](X\times
Y)$. By \ref{mthm} it is then enough to show that there is a
distinguished triangle in $\D[b](X\times X)$
\begin{equation}\label{dist}
\diag_*\so_X(-X)[1]\to\Ker\FMcomp\Kerd\mor{\mult{\Ker}}\sod[X]\to
\diag_*\so_X(-X)[2],
\end{equation}
which is done in the proof of \cite[Cor. 11.4]{H} (notice that, by
\ref{FMadj} and \ref{FMcomp}, $\FM{\Kerd}\iso f_*$ and
$\FM{\Ker\FMcomp\Kerd}\iso f^*\comp f_*$).
\end{proof}

\begin{rema}
More generally, if $f\colon X\to Y$ is a morphism such that
$\cone{f\mrs}\iso\ds_Y[c]$, then it can be proved that
$\ST{f^*\exc{0}}\comp\cdots\comp\ST{f^*\exc{m}}\iso \FM{\cone{(f\times
f)^*\sod[Y]\to\sod[X]}}$. If moreover $f$ is flat, then by flat base
change $(f\times f)^*\sod[Y]\iso\so_{X\times_YX}$: when $Y=\Ps^n$ and
$\exc{i}=\so_{\Ps^n}(1)$ this result had already been proved in
\cite{AHK}.
\end{rema}

\begin{rema}
If $X\subset\Ps=\Ps(\w_0,\dots,\w_n)$ is a (Calabi-Yau)
hypersurface of degree $\sw{\w}:=\w_0+\cdots+\w_n$ and
$(\exc{0},\dots,\exc{m})=(\so_{\Ps}(1),\dots,\so_{\Ps}(\sw{\w}))$,
then \ref{pullback} reduces to
$\ST{\so_X(1)}\comp\cdots\comp\ST{\so_X(\sw{\w})}\iso
\T{\so_X(-\sw{\w})[2]}$, which is in fact equivalent to
\cite[Thm. 1.1]{CKa}, namely $(\FM{\s{G}})^{\comp\sw{\w}}\iso(-)[2]$,
where $\s{G}:= \cone{\so_{X\times X}(1,0)\to\diag_*\so_X(1)}$. To see
this, note that $\FM{\s{G}}\iso\T{\so_X(1)}\comp\ST{\so_X}$, hence
using \ref{STcomp} below we obtain
\[(\FM{\s{G}})^{\comp\sw{\w}}\iso\ST{\so_X(1)}\comp\cdots\comp
\ST{\so_X(\sw{\w})}\comp\T{\so_X(\sw{\w})}.\]
\end{rema}

\begin{lemm}\label{STcomp}
If $\s{F}\in\D[b](X)$ and $\funn$ is a Fourier-Mukai autoequivalence
of $\D[b](X)$, then $\funn\comp\ST{\s{F}}\iso
\ST{\funn(\s{F})}\comp\funn$.
\end{lemm}

\begin{proof}
See \cite{ST} or \cite[Lemma 8.21]{H}.  
\end{proof}

\begin{coro}\label{pushout}
If $g\colon Y\mono X$ is the inclusion of a hypersurface such that
$g^*\ds_X\iso\so_Y$, then $\ST{g_*\exc{0}}\comp\cdots\comp\ST{g_*\exc{m}}
\iso\T{\so_X(Y)}$.
\end{coro}

\begin{proof}
$g_*\iso\FM{\Ker}$, where $\Ker:=(g,\id_Y)_*\so_Y\in\D[b](X\times
Y)$. By \ref{mthm} it is enough to prove that there is a
distinguished triangle in $\D[b](X\times X)$
\[\diag_*\so_X(Y)[-1]\to\Ker\FMcomp\Kerd\mor{\mult{\Ker}}\sod[X]\to
\diag_*\so_X(Y),\]
which can be done with a technique similar to the one used for the
proof of \eqref{dist}. As an indication that this is true,
observe that, always by \ref{FMadj} and \ref{FMcomp}, $\FM{\Kerd}\iso
g^!$ and $\FM{\Ker\FMcomp\Kerd}\iso g_*\comp g^!$. Now, for every
$\s{F}\in\D[b](X)$ there is a natural isomorphism 
\[g_*g^!\s{F}\iso g_*(g^*\s{F}\otimes\ds_Y[-1])\iso\s{F}\otimes
g_*\ds_Y[-1]\iso \s{F}\otimes g_*\so_Y(Y)[-1],\]
hence $g_*\comp g^!\iso\T{g_*\so_Y(Y)[-1]}$. In fact one can prove
that $\Ker\FMcomp\Kerd\iso\diag_*g_*\so_Y(Y)[-1]$, and that the above
triangle can be identified with the image through $\diag_*$ of the
distinguished triangle in $\D[b](X)$
\[\so_X(Y)[-1]\to g_*\so_Y(Y)[-1]\to\so_X\to\so_X(Y)\]
induced by the short exact sequence 
$0\to\so_X\to\so_X(Y)\to g_*\so_Y(Y)\to0$.
\end{proof}

Finally we want to show how \ref{pushout} can be used in a setup like
that of \cite{K1}, respectively \cite{K2}, where $X$ is a crepant
resolution of the singularity $\C^2/\Z_3$, respectively
$\C^3/\Z_5$. Actually we will assume (as always) that $X$ is proper,
but it would be not difficult to see that our arguments can be extended to
the setting of varieties or stacks which are not necessarily proper,
by working with derived categories of coherent sheaves with compact
supports. 

If $g\colon Y\mono X$ is the inclusion morphism and
$\s{F}\in\D[b](Y)$, in the following we will write for simplicity
$\s{F}$ instead of $g_*\s{F}\in\D[b](X)$.

First let $\dim(X)=2$ and let $C_i$ (for $i=1,\dots,4$) be divisors in
$X$ such that $C_3$ and $C_4$ are $(-2)$-curves. Assume also the
following intersection relations
\begin{equation}\label{inter1}
C_2\cdot C_3=C_1\cdot C_4=1,\qquad C_1\cdot C_3=C_2\cdot C_4=0
\end{equation}
and the following linear equivalence relations of divisors
\begin{equation}\label{linear1}
C_3+2C_2\sim C_1,\qquad C_4+2C_1\sim C_2.
\end{equation}
Taking into account that $(\so_{\Ps^1},\so_{\Ps^1}(1))$ is a full
exceptional sequence in $\D[b](\Ps^1)$, \ref{pushout} implies that for
$i=3,4$
\begin{equation}\label{m1}
\ST{\so_{C_i}}\comp\ST{\so_{C_i}(1)}\iso\T{\so_X(C_i)}.
\end{equation}
Setting $\M{2}:=\ST{\so_{C_3}}\comp\T{\so_X(C_2)}$ and
$\M{3}:=\ST{\so_{C_4}}\comp\M{2}$, in \cite{K1} it was proved that
$\M{2}^2\iso\T{\so_X(C_1)}$ and conjectured that $\M{3}^3\iso\id$ in
$\Aut(\D[b](X))$. We are going to show that both results follow easily
from \eqref{m1}. Indeed, using also \ref{STcomp}, \eqref{inter1} and
\eqref{linear1}, we find
\begin{multline*}
\M{2}^2=
\ST{\so_{C_3}}\comp\T{\so_X(C_2)}\comp\ST{\so_{C_3}}\comp\T{\so_X(C_2)}
\iso\ST{\so_{C_3}}\comp\ST{\T{\so_X(C_2)}(\so_{C_3})}\comp\T{\so_X(C_2)}
\comp\T{\so_X(C_2)} \\
\iso\ST{\so_{C_3}}\comp\ST{\so_{C_3}(1)}\comp\T{\so_X(2C_2)}\iso
\T{\so_X(C_3)}\comp\T{\so_X(2C_2)}\iso\T{\so_X(C_3+2C_2)}\iso
\T{\so_X(C_1)}.
\end{multline*}
In a completely similar way one proves that
$\Ma{2}:=\ST{\so_{C_4}}\comp\T{\so_X(C_1)}$ satisfies
$\Ma{2}^2\iso\T{\so_X(C_2)}$. Since $\M{3}(\so_{C_4})\iso\so_{C_3}$
(see \cite[Section 4.2]{K1}), we obtain
\begin{multline*}
\M{3}^2=\M{3}\comp\ST{\so_{C_4}}\comp\M{2}\iso
\ST{\M{3}(\so_{C_4})}\comp\M{3}\comp\M{2} \\
\iso\ST{\so_{C_3}}\comp\ST{\so_{C_4}}\comp\M{2}^2\iso
\ST{\so_{C_3}}\comp\ST{\so_{C_4}}\comp\T{\so_X(C_1)}=
\ST{\so_{C_3}}\comp\Ma{2}.
\end{multline*}
On the other hand, $\M{3}\iso\Ma{2}\comp\T{\so_X(-C_1)}\comp\M{2}\iso
\Ma{2}\comp\M{2}^{-1}$,
and we conclude
\[\M{3}^3\iso\ST{\so_{C_3}}\comp\Ma{2}^2\comp\M{2}^{-1}
\iso\ST{\so_{C_3}}\comp\T{\so_X(C_2)}\comp\M{2}^{-1}=
\M{2}\comp\M{2}^{-1}\iso\id.\]

Now let $\dim(X)=3$ and let $D_i$ (for $i=1,\dots,5$) be divisors in
$X$ such that $D_4\iso\Ps^2$, $D_5\iso\F_3$ and $K_X\cdot D_4=K_X\cdot
D_5=0$. Denoting by $h$ the class of a hyperplane section in $D_4$, by
$f$ the class of a fibre in $D_5$ and by $s$ the class of the $-3$
section in $D_5$, assume also the following intersection relations
\begin{equation}\label{inter2}
D_4\cdot D_1=h,\quad D_4\cdot D_2=0,\qquad D_5\cdot D_1=f,\quad
D_5\cdot D_2=s+3f
\end{equation}
and the following linear equivalence relations of divisors
\begin{equation}\label{linear2}
D_4+3D_1\sim D_2,\qquad D_5+2D_2\sim D_1.
\end{equation}
Using the full exceptional sequences 
$(\so_{D_4}(h),\so_{D_4}(2h),\so_{D_4}(3h))$ in $\D[b](D_4)$ and
$(\so_{D_5}(-f),\so_{D_5},\so_{D_5}(s+2f),\so_{D_5}(s+3f))$ in
$\D[b](D_5)$, by \ref{pushout} we have
\begin{gather}
\ST{\so_{D_4}(h)}\comp\ST{\so_{D_4}(2h)}\comp\ST{\so_{D_4}(3h)}
\iso\T{\so_X(D_4)},\label{m3} \\
\ST{\so_{D_5}(-f)}\comp\ST{\so_{D_5}}\comp\ST{\so_{D_5}(s+2f)}
\comp\ST{\so_{D_5}(s+3f)}\iso\T{\so_X(D_5)}.\label{m2}
\end{gather}
Setting $\MM{2}:=\T{\so_X(-D_1)}\comp\ST{\so_{D_5}}\comp
\T{\so_X(D_1)}\comp\ST{\so_{D_5}}\comp\T{\so_X(D_2)}$,
$\MM{3}:=\T{\so_X(D_1)}\comp\ST{\so_{D_4}}$ and
$\MM{5}:=\ST{\so_{D_4}}\comp\MM{2}$, in \cite{K2} it was conjectured
that $\MM{2}^2\iso\T{\so_X(D_1)}$, $\MM{3}^3\iso\T{\so_X(D_2)}$ and
$\MM{5}^5\iso\id$ in $\Aut(\D[b](X))$. Here we prove only the first
two relations (then it is not difficult to deduce also the last
one, with manipulations similar to those with which $\M{3}^3\iso\id$
was obtained above using $\M{2}^2\iso\T{\so_X(C_1)}$). To this
purpose, observe that by \ref{STcomp} and \eqref{inter2} we have
$\T{\so_X(-D_1)}\comp\ST{\so_{D_5}}\iso
\ST{\so_{D_5}(-f)}\comp\T{\so_X(-D_1)}$, hence $\MM{2}\iso
\ST{\so_{D_5}(-f)}\comp\ST{\so_{D_5}}\comp\T{\so_X(D_2)}$. Then, using
also \eqref{m2} and \eqref{linear2} we get
\begin{multline*}
\MM{2}^2\iso\ST{\so_{D_5}(-f)}\comp\ST{\so_{D_5}}\comp\T{\so_X(D_2)}
\comp\ST{\so_{D_5}(-f)}\comp\ST{\so_{D_5}}\comp\T{\so_X(D_2)} \\
\iso\ST{\so_{D_5}(-f)}\comp\ST{\so_{D_5}}\comp\ST{\so_{D_5}(s+2f)}
\comp\ST{\so_{D_5}(s+3f)}\comp\T{\so_X(D_2)}\comp\T{\so_X(D_2)} \\
\iso\T{\so_X(D_5)}\comp\T{\so_X(2D_2)}\iso\T{\so_X(D_5+2D_2)}
\iso\T{\so_X(D_1)}.
\end{multline*}
Similarly, using \eqref{m3} we find
\begin{multline*}
\MM{3}^3=\T{\so_X(D_1)}\comp\ST{\so_{D_4}}\comp\T{\so_X(D_1)}\comp
\ST{\so_{D_4}}\comp\T{\so_X(D_1)}\comp\ST{\so_{D_4}} \\
\iso\ST{\so_{D_4}(h)}\comp\ST{\so_{D_4}(2h)}\comp\ST{\so_{D_4}(3h)}
\comp\T{\so_X(D_1)}\comp\T{\so_X(D_1)}\comp\T{\so_X(D_1)} \\
\iso\T{\so_X(D_4)}\comp\T{\so_X(3D_1)}\iso\T{\so_X(D_4+3D_1)}\iso
\T{\so_X(D_2)}.
\end{multline*}

\appendix

\section{Triangulated categories and exceptional sequences}

In this appendix we collect some definitions and properties about
triangulated categories, semiorthogonal decompositions and exceptional
sequences. We refer to \cite{V}, \cite{Bo} and \cite{GK} for a
thorough treatment of these subjects, including proofs of the results
which are only stated here.

A {\em triangulated category} is an additive category endowed with an
additive automorphism, called {\em shift functor} and denoted by $[1]$
(its $n^{\rm th}$ power is denoted by $[n]$ for every $n\in\Z$), and
with a family of {\em distinguished triangles} of the form
$A\mor{f}B\mor{g}C\mor{h}A[1]$, subject to a list of axioms, usually
denoted by (TR1)-(TR4).  Let $\cat{T}$ be a triangulated category. We
call {\em cone} of a morphism $f\colon A\to B$ of $\cat{T}$ an object
$\cone{f}$ fitting into a distinguished triangle
$A\mor{f}B\to\cone{f}\to A[1]$: it is a fundamental property of
triangulated categories that every morphism admits a cone, which is
unique up to (non canonical) isomorphism.  The following simple fact
is frequently used in the paper: given two morphisms $f\colon A\to B$
and $g\colon B\to C$ in $\cat{T}$, we have $\cone{g\comp
f}\iso\cone{f}$ if $g$ is an isomorphism and $\cone{g\comp
f}\iso\cone{g}$ if $f$ is an isomorphism.  We recall that by axiom
(TR3), given a commutative diagram of continuous arrows whose rows are
distinguished triangles
\[\xymatrix{A \ar[rr]^-f \ar[d]^{a} & & B \ar[rr]^-g \ar[d]^b & &
C \ar[rr]^-h \ar@{-->}[d]^c & & A[1] \ar[d]^{a[1]} \\ 
A' \ar[rr]_-{f'} & & B' \ar[rr]_-{g'} & & C' \ar[rr]_{h'} & & A[1]}\]
there exists (not unique in general) a dotted arrow $c$ keeping the
diagram commutative.  Moreover, axiom (TR4) implies that, if
$f=\id_A$, then $c$ can be chosen with the additional property that
$\cone{c}\iso\cone{b}$. We ignore if (always when $f=\id_A$) every
morphism $c$ making the diagram commute satisfies
$\cone{c}\iso\cone{b}$, but we will see later that this is true with
some assumptions on $\cat{T}$.

An {\em exact} functor between two triangulated categories $\cat{T}$
and $\cat{T}'$ is given by a functor $\fun\colon\cat{T}\to\cat{T}'$
together with an isomorphism of functors $\siso{\fun}\colon\fun\comp[1]\to
[1]\comp\fun$, such that for every distinguished triangle
$A\mor{f}B\mor{g}C\mor{h} A[1]$ of $\cat{T}$, 
\[\fun(A)\mor{\fun(f)}\fun(B)\mor{\fun(g)}\fun(C)
\mor{\siso{\fun}(A)\comp\fun(h)}\fun(A)[1]\]
is a distinguished triangle of $\cat{T}'$.

Given $S\subseteq\ob(\cat{T})$ (the collection of objects of
$\cat{T}$), we denote by $\trspan{S}$ the smallest strictly full
triangulated subcategory of $\cat{T}$ containing $S$, and by $S\ort$
the (right) orthogonal of $S$, namely the full subcategory of
$\cat{T}$ whose objects are those $A\in\cat{T}$ such that
$\Hom(B,A[n])=0$ for every $B\in S$ and for every $n\in\Z$ (note that
$S\ort$ is a strictly full triangulated subcategory of $\cat{T}$); if
$\cat{C}$ is a subcategory of $\cat{T}$, we write $\cat{C}\ort$ for 
$\ob(\cat{C})\ort$. 
A strictly full triangulated subcategory $\cat{T}'$ of $\cat{T}$ is
{\em admissible} if the inclusion functor $\cat{T}'\to\cat{T}$ admits
left and right adjoints.
A sequence $(\cat{T}_0,\dots,\cat{T}_m)$ of admissible subcategories
of $\cat{T}$ is {\em semiorthogonal} if
$\cat{T}_i\subseteq\cat{T}_j\ort$ for $0\le i<j\le m$; if moreover
$\cat{T}=\trspan{\ob(\cat{T}_0)\cup\cdots\cup\ob(\cat{T}_m)}$, then
$(\cat{T}_0,\dots,\cat{T}_m$) is called a {\em semiorthogonal
decomposition} of $\cat{T}$.

\begin{lemm}\label{semiort}
If $(\cat{T}_0,\cat{T}_1)$ is a semiorthogonal decomposition of
$\cat{T}$, then for every object $A$ of $\cat{T}$ there exists (unique
up to isomorphism) a distinguished triangle $A_1\to A\to A_0\to
A_1[1]$ with $A_i\in\cat{T}_i$.
\end{lemm}

From now on we assume that $\cat{T}$ is a $\K$-linear triangulated
category and that $\dim_{\K}(\bigoplus_{k\in\Z}\Hom(A,B[k])<\infty$
for every objects $A$ and $B$ of $\cat{T}$ (this condition is clearly
satisfied when $\cat{T}=\D[b](X)$ with $X$ a stack).

\begin{prop}\label{cone!}
Given a commutative diagram in $\cat{T}$
\[\xymatrix{A \ar[rr]^-f \ar[d]^{\id} & & B \ar[rr]^-g \ar[d]^b & &
C \ar[rr]^-h \ar[d]^c & & A[1] \ar[d]^{\id} \\ 
A \ar[rr]_-{f'} & & B' \ar[rr]_-{g'} & & C' \ar[rr]_{h'} & & A[1]}\]
whose rows are distinguished triangles, we have $\cone{c}\iso\cone{b}$.
\end{prop}

\begin{proof}
It follows from \cite[Prop. 2.1]{CKu}, since $\Hom_{\cat{T}}(C',C')$ is
a finite-dimensional $\K$-algebra.
\end{proof}

\begin{defi}
$E\in\cat{T}$ is {\em exceptional} if $\Hom_{\cat{T}}(E,E)\iso\K$ and
$\Hom_{\cat{T}}(E,E[k])=0$ for $k\ne0$.  $(E_0,\dots,E_m)$ is an {\em
exceptional sequence} in $\cat{T}$ if each $E_i$ is an exceptional
object and $\Hom_{\cat{T}}(E_j,E_i[k])=0$ for $0\le i<j\le m$ and for
every $k\in\Z$; if moreover $\cat{T}= \trspan{E_0,\dots,E_m}$, then
$(E_0,\dots,E_m)$ is called a {\em full exceptional sequence}.
\end{defi}

\begin{rema}
If $(E_0,\dots,E_m)$ is a (full) exceptional sequence in $\cat{T}$,
then $\trspan{E_0,\dots,E_m}$ is an admissible subcategory of
$\cat{T}$ and $(\trspan{E_0},\dots,\trspan{E_m})$ is a semiorthogonal
sequence (decomposition).
\end{rema}

Given a full exceptional sequence $(E_0,\dots,E_m)$, for $0\le j\le
i\le m$ we define inductively objects $\ldual{j}{E_i}$ by
$\ldual{0}{E_i}:=E_i$ and
\[\ldual{j}{E_i}:=
\cone{\bigoplus_{k\in\Z}\Hom_{\cat{T}}(E_{i-j}[k],\ldual{j-1}{E_i})
\otimes_{\K}E_{i-j}[k]\to\ldual{j-1}{E_i}}\] 
for $0<j\le i$ (the morphism being the natural one). Setting
$E'_{i}:=\ldual{i}{E_i}$ for $0\le i\le m$, it is easy to prove that
$(E'_m,\dots,E'_0=E_0)$ is again a full exceptional sequence, called
the (right) dual of $(E_0,\dots,E_m)$ because of the following result.

\begin{lemm}\label{ort}
If $(E_0,\dots,E_m)$ is a full exceptional sequence and
$(E'_m,\dots,E'_0)$ is the dual exceptional sequence, then
$\Hom_{\cat{T}}(E_i,E'_j[k])\iso\K^{\delta_{i,j}\delta_{k,0}}$ for $0\le i,j\le
m$ and for every $k\in\Z$.
\end{lemm}

\end{document}